\documentclass[leqno,12pt]{amsart}
\usepackage{a4,latexsym,amssymb,amsfonts,amsthm}
\usepackage{color}
\usepackage{soul}

\RequirePackage{pifont}

\renewcommand{\Re}{\operatorname{Re}}
\renewcommand{\Im}{\operatorname{Im}}

\newcommand{\sfrac}[2]{{#1 / #2}}  

\newtheorem*{theorem*}{Theorem}
\newtheorem*{corollary*}{Corollary}

\newtheorem*{theorema}{Theorem A}
\newtheorem*{theoremb}{Theorem B}
\newtheorem*{theoremc}{Theorem C}
\newtheorem*{theoremd}{Theorem D}

\newtheorem{theorem}{Theorem}
\newtheorem{lemma}[theorem]{Lemma}
\newtheorem{proposition}[theorem]{Proposition}
\newtheorem{corollary}[theorem]{Corollary}

\theoremstyle{definition}
\newtheorem{remark}[theorem]{Remark}

\numberwithin{theorem}{section}
\numberwithin{equation}{section}

\newcommand{\fn}[1]{\operatorname{\mathit{#1}}}

\newcommand\RR{{\mathbb{R}}}
\newcommand\CC{{\mathbb{C}}}
\newcommand\cS{{\mathcal{S}}}
\newcommand\cH{{\mathcal{H}}}

\newcommand\lieg{\mathfrak{g}}
\newcommand{\wrt}{\,d}

\newcommand{\lnorm}{\left \|}
\newcommand{\rnorm}{\right \|}
\newcommand{\biglnorm}{\bigl \|}
\newcommand{\bigrnorm}{\bigr \|}

\newcommand{\lopnorm}{\left.{|\!|\!|}}
\newcommand{\ropnorm}{{|\!|\!|}\right.}

\newcommand{\labs}{\left |}
\newcommand{\rabs}{\right |}
\newcommand{\biglabs}{\bigl |}
\newcommand{\bigrabs}{\bigr |}

\newcommand{\absdot}{\left|{\,{\cdot}\,}\right|}

\newcommand{\lip}{\left <}
\newcommand{\rip}{\right >}

\newcommand{\lpar}{\left( }
\newcommand{\rpar}{\right) }
\newcommand{\biglpar}{\bigl( }
\newcommand{\bigrpar}{\bigr) }
\newcommand{\Biglpar}{\Bigl( }
\newcommand{\Bigrpar}{\Bigr) }
\newcommand{\bigglpar}{\biggl( }
\newcommand{\biggrpar}{\biggr) }

\newcommand{\rist}{\Bigr|_}
\newcommand{\bigrist}{\bigr|_}
\newcommand{\Bigrist}{\Bigr|_}

\begin{document}

\title[Inequalities on stratified groups]
{Hardy and uncertainty inequalities
\\ on stratified Lie groups}
\author[P. Ciatti]{Paolo Ciatti}
\address{Dipartimento di Metodi e Modelli Matematici per le Scienze Applicate,\newline\indent
Universit\`a di Padova\\ Via Trieste 63\\ 35121 Padova\\ Italy}
\email{ciatti@dmsa.unip.it}
\author[M. G. Cowling]{Michael G.\ Cowling}
\address{School of Mathematics and Statistics\\University of New South Wales,\newline\indent
UNSW Sydney 2052\\ Australia}
\email{m.cowling@unsw.edu.au}
\author[F. Ricci]{Fulvio Ricci}
\address{Scuola Normale Superiore\\
Piazza dei Cavalieri 7\\ 56126 Pisa\\ Italy}
\email{fricci@sns.it}
\thanks{The first author acknowledges Australian Research Council support (grant number DP110102488).}
\keywords{Stratified group, uncertainty principle, Hardy's inequality, Heisenberg's inequality}
\subjclass{primary 42B37; secondary 43A80}

\begin{abstract}
We prove various Hardy-type and uncertainty inequalities on a stratified Lie group $G$.
In particular, we show that the operators $T_\alpha: f \mapsto \absdot^{-\alpha} L^{-\alpha/2} f$, where $\absdot$ is a homogeneous norm, $0 < \alpha < Q/p$, and $L$ is the sub-Laplacian, are bounded on the Lebesgue space $L^p(G)$.
As consequences, we estimate the norms of these operators sufficiently precisely to be able to differentiate and prove a logarithmic uncertainty inequality.
We also deduce a general version of the Heisenberg--Pauli--Weyl inequality, relating the $L^p$ norm of a function $f$ to the $L^q$ norm of $\absdot^\beta f$ and the $L^r$ norm of $L^\sfrac\delta2 f$.
\end{abstract}

\maketitle

\section{Introduction}
In 1920, in a paper on Fourier series, Hardy stated the following integral inequality; he published a proof five years later.
Given a nonnegative (measurable) function $f$ on $\RR^+$, let $F(x) = \int_0^x f(t) \,dt$.
Then
\[
\int_0^\infty \labs \frac{ F(x) }{x} \rabs^p \,dx \leq \lpar \frac{p}{p-1}\rpar^{p} \int_0^\infty \labs f(x) \rabs^p \,dx \ ,
\]
when $p > 1$ and the right hand side is finite.
See \cite[Theorem 327]{HLP} for the history of this inequality.
Of course, this is much the same as the inequality
\[
\int_0^\infty \labs \frac{ f(x) }{x} \rabs^p \,dx \leq \lpar \frac{p}{p-1}\rpar^{p} \int_0^\infty \labs \frac{df(x)}{dx} \rabs^p \,dx \ .
\]

We might also replace $f(x)$ by $x f(x)$ and interpret this inequality as a statement about the boundedness of the operator $T$, given by
\[
Tf(x) = \frac{d(xf(x))}{dx} \ ,
\]
or of the dual operator $T^*$:
\[
T^*f(x) = x \frac{df(x)}{dx} \ .
\]

In this paper, we will focus on generalizations of $T^*$ rather than of $T$.

Hardy's inequality, which is related to inequalities of Rellich and of Sobolev, has been extended in many ways; for instance, Davies and Hinz \cite{DavHin} showed that
\[
\int_{\RR^n}  \frac{ \labs f(x) \rabs^p }{ \labs x \rabs^p} \,dx \leq \lpar \frac{p}{n-p} \rpar^p \int_{\RR^n} \labs \nabla f (x) \rabs^p \,dx
\]
for all $f \in C^\infty_c(\RR^n)$, when $n > p$.

In 1927, Heisenberg presented a heuristic argument for his famous uncertain principle; the mathematical details were provided by Pauli and Weyl.
In mathematical language, this inequality states that, if $f$ is a suitable function on $\RR$, then
\begin{equation}\label{ineq:HPW-in-R}
\lpar \int_{\RR} \labs f(x) \rabs^2 \,dx \rpar ^{2}
\leq  2 \lpar \int_{\RR} \labs x f(x) \rabs^2 \,dx \rpar \lpar \int_{\RR} \labs \frac{d f(x)}{dx} \rabs^2 \,dx \rpar .
\end{equation}

Many inequalities combining features of both Hardy's inequality and Heisenberg's uncertainty principle are known; for instance,
\begin{equation}\label{ineq:Hardy-in-Rn}
\lpar \int_{\RR^n} \labs f(x) \rabs^2 \,dx \rpar^2
 \leq  C(n,\alpha) \lpar \int_{\RR^n} \biglabs \labs x \rabs^{\alpha} f(x) \bigrabs^2 \,dx \rpar \lpar \int_{\RR^n} \labs L^{\sfrac{\alpha}{2}} f(x) \rabs^2 \,dx \rpar ,
\end{equation}
where $L$ is minus the Laplacian.
An equivalent form of \eqref{ineq:Hardy-in-Rn} is
\begin{equation*}
\lpar \int_{\RR^n} \labs f(x) \rabs^2 \,dx \rpar^2
 \leq  C(n,\alpha) \lpar \int_{\RR^n} \biglabs \labs x \rabs^{\alpha} f(x) \bigrabs^2 \,dx \rpar
 		\lpar \int_{\RR^n} \biglabs \labs \xi \rabs^{\alpha} \hat f(\xi) \bigrabs^2 \,d\xi \rpar
\end{equation*}
(using the factor $\exp(- i x \cdot \xi)$ in the definition of the Fourier transform $\hat f$ of $f$).
If $C(n,\alpha)$ is the best constant in this inequality (see  \cite{Beck1}), then $C(n,\alpha) \to 1$ as $\alpha \to 0+$, and $C(n, \alpha)$ is right differentiable (as a function of $\alpha$) at $0$, with right derivative $D(n)$, say.
As a consequence, by differentiating, we obtain the logarithmic uncertainty inequality
\begin{equation}\label{ineq:logarithmic-in-Rn-Fourier}
- \frac{D(n)}{2} \int_{\RR^n} \labs f(x) \rabs^2 \,dx
\leq  \int_{\RR^n} \log \labs x\rabs  \labs f(x) \rabs^2 \,dx  + \int_{\RR^n} \log \labs\xi\rabs | \hat f(\xi) |^2 \,d\xi \ .
\end{equation}

Alternatively, we could write
\begin{equation}\label{ineq:logarithmic-in-Rn}
- \frac{D(n)}{2} \int_{\RR^n} \labs f(x) \rabs^2 \,dx
\leq  \int_{\RR^n} \log \labs x\rabs  \labs f(x) \rabs^2 \,dx  + \int_{\RR^n} \Re \lpar \lpar \log L^{\sfrac{1}{2}} f(x) \rpar \bar f(x) \rpar \,dx \ .
\end{equation}
It is known that the Heisenberg--Pauli--Weyl inequality \eqref{ineq:Hardy-in-Rn} may be recovered from \eqref{ineq:logarithmic-in-Rn-Fourier}.
For more on these inequalities, including how to deduce \eqref{ineq:Hardy-in-Rn} from \eqref{ineq:logarithmic-in-Rn-Fourier}, generalizations to $L^p$ spaces, and the best constants therein, see Beckner \cite{Beck1, Beck2}.

These inequalities have been generalized in many ways.
Uncertainty inequalities have been extended to environments such as Lie groups and manifolds;  see Folland and Sitaram \cite{FSit} for more information about older work.
More recent work on the general topic of uncertainty principles on groups and manifolds includes \cite{CRS, Tao, VSC}.
Earlier work on the problems that we treat here includes \cite{BCX, GaroLanco}.
Less is known about logarithmic inequalities in more general contexts, and one of the main aims of our paper is to generalize the inequality \eqref{ineq:logarithmic-in-Rn} to $L^p$ spaces on stratified Lie groups.
This entails generalizing \eqref{ineq:Hardy-in-Rn} to $L^p$ spaces on stratified groups and controlling the constants as $\alpha \to 0+$.

For completeness, we mention
that there are two types of logarithmic inequality that are related to
Heisenberg--Pauli--Weyl inequalities. Besides \eqref{ineq:logarithmic-in-Rn-Fourier}, there are  also inequalities arising from the Hausdorff-Young inequality by differentiating in the Lebesgue index, which introduces a factor of $\log |f|$.  These were
apparently first noticed by Hirschman~\cite{Hir}; other important work on the
topic includes \cite{Beck3, Oz-Prz}, and, in the
context of nilpotent groups,~\cite{BFM}.
In this article, we deal with inequalities where the weight functions are differentiated,
introducing factors such as $\log |x|$.

In this paper, we take a stratified Lie group $G$ of homogeneous dimension $Q$, with a positive hypoelliptic sub-Laplacian $L$ (more on these terms later).

In $\RR^n$,  the negative fractional powers of the Laplacian  are given by convolution with a negative power of the euclidean norm; in general stratified groups, the situation is complicated by the fact that each power $L^{-\alpha/2}$ of the sub-Laplacian involves a different homogeneous norm $\absdot_\alpha$.
For the detailed analysis of the behaviour of $L^{-\alpha/2}$ as $\alpha$ tends to 0, which we develop  in Section \ref{sec:homo}, it is convenient to use a special homogeneous norm, written $\absdot_0$, which is a limit of the norms $\absdot_\alpha$.

 In Section 3, we study the Hardy operator $T_\alpha$ associated to a homogeneous norm $\absdot$ by
\[
T_\alpha f =   \absdot^{-\alpha}  L^{-\alpha/2} f
\]
for all $f \in C^\infty_c(G)$. Using the Schur criterion, we derive Theorem \ref{Talpha} and Corollary \ref{finalmente}, which combine to give the following result.

\begin{theorema}
Suppose that $\absdot$ is a homogeneous norm on $G$, that $1 < p < \infty$ and that $0 < \alpha < \sfrac{Q}{p}$.
Then  the operator $T_{\alpha}$ extends uniquely to a bounded operator  on $L^p(G)$.
For the particular homogeneous norm $\absdot_0$, the operator norm $\lopnorm T_\alpha \ropnorm_{p,p}$ satisfies
\[
\lopnorm T_\alpha \ropnorm_{p,p} \leq 1 + C \alpha + O(\alpha^2)\ .
\]
\end{theorema}

The next two main results follow from Theorem A.
First, we combine  Theorem A with H\"older's inequality and differentiate to obtain the following statement.
\begin{theoremb}
Suppose that $\absdot$ is a homogeneous norm on $G$ and $1 < p < \infty$.
There is a constant $C$ such that
\begin{equation*}
\int_{G}  (\log \labs x\rabs) \labs f (x) \rabs^{p} \wrt x
	+ \int_{G} \Re \biglpar (\log L^\sfrac12 f)(x)\,\overline{f (x) } \bigrpar  \labs f (x) \rabs^{p - 2} \wrt x
 \ge C \lnorm f \rnorm_p^p
\end{equation*}
for all  $f\in C^{\infty}_c(G)$.
\end{theoremb}

Second, in Sections \ref{sec:HPW_C} and \ref{sec:HPW_R}, we explore Heisenberg--Pauli--Weyl type inequalities.
In Section \ref{sec:HPW_C}, we take Theorem A as out starting point, while in Section \ref{sec:HPW_R}, we give an alternative method of attack.
Our main result in this direction  is as general as one might hope for, although we lose control of the constants.

\begin{theoremc}
Suppose that $\absdot$ is a homogeneous norm, that $\beta>0$, $\delta>0$, $p>1$, $s\geq 1$,  $r > 1$, and that
\begin{equation}\label{pqr}
\frac{\beta+\delta}p = \frac\delta s+\frac\beta r \ .
\end{equation}
Then
\begin{equation}\label{HPWeq2}
\lnorm f \rnorm_p
\le C  \biglnorm \absdot^\beta f\bigrnorm_s^\sfrac\delta{(\beta+\delta)}\biglnorm  L^\sfrac\delta2 f\bigrnorm_r^\sfrac\beta{(\beta+\delta)}
\end{equation}
for all $f \in C^\infty_c(G)$.
\end{theoremc}

A simple dilation and homogeneity argument shows that the only possible indices for which we could hope to prove an inequality like \eqref{HPWeq2} are those which satisfy \eqref{pqr}.

The proof of Theorem C requires an extension of the classical Landau--Kolmogorov inequality \cite{KL}, which may be of independent interest, and so we state it explicitly here.

\begin{theoremd}
Suppose that $0 \leq \theta \leq 1$ and that $\alpha \geq 0$.
If $1 < p, q, r < \infty$, and
\[
\frac{1}{p} = \frac{\theta}{q} + \frac{1-\theta}{r} \ ,
\]
then
\[
\lnorm L^{\sfrac{\alpha}{2}} f  \rnorm_p \leq C \lnorm  L^{\sfrac{\alpha}{2\theta}} f \rnorm_q^\theta \lnorm f \rnorm_r^{1-\theta}
\]
for all $f \in C^\infty_c(G)$.
\end{theoremd}

The proof of Theorem B is given in Section \ref{sec:log}.
Theorems C and D are proved in Section \ref{sec:HPW_C}.

Given the very general context in which we work, we do not make any attempt to assign explicit values to the constants appearing in the above inequalities.
We believe that this is a challenging problem in special cases, and almost impossible in general.

The symbol $C$ will be used throughout to denote an undefined constant which may vary from one line to the next. When necessary we specify, possibly with a subscript, which parameters the value of $C$ may depend on.

Any dependence of constants on the group $G$ and the sub-Laplacian $L$ is ignored.

\section{Riesz potentials} \label{sec:homo}

Consider a connected, simply connected, nilpotent Lie group $G$ of dimension $n$.
As the exponential map is bijective, we may identify $G$ with its Lie algebra $\lieg$.
With this identification, the Haar measure of $G$ is given by the Lebesgue measure in the vector space $\lieg$.
We may define the Schwartz space $\cS(G)$ on $G$ similarly.

Assume that $G$ is \emph{stratified}, that is, the Lie algebra $\lieg$ of $G$ has a vector space direct sum decomposition
\begin{equation*}
\lieg = \lieg_1 \oplus \lieg_2 \oplus \dots \oplus \lieg_m\ ,
\end{equation*}
where $\lieg_{j+1}=[\lieg_1,\lieg_{j}]$ for $j = 1, \dots, m$ (we set $\lieg_{m+1} =\{0\}$).
For $r \in \RR^+$, the \emph{dilation} $\delta_{r}$ is the automorphism of $\lieg$ given by scalar multiplication by $r^j$ on $\lieg_j$ for each $j$.
The integer
 \begin{equation*}
 Q = \dim \lieg_1 + 2 \dim \lieg_2 + \dots + m\dim \lieg_m
 \end{equation*}
 is the \emph{homogeneous dimension} of $G$.

Fix a basis $\{X_1, \dots, X_l\}$ of  $\lieg_1$.
The associated  \emph{sub-Laplacian} $L$ is defined by
\begin{equation*}
L = - X_1^2 - X_2^2 - \dots -X_l^2\ .
\end{equation*}
It is well known that $L$ is  hypoelliptic \cite{Hor67},  and positive and essentially self-adjoint on $L^2(G)$ \cite{HJL85}.

We now construct kernels corresponding to certain negative powers of the sub-Laplacian, slightly improving on the results in Folland \cite{folland}.
By a theorem of G. A. Hunt \cite{GHunt}, the semigroup $(e^{-tL})_{t > 0}$ generated by $L$ (or more precisely, by its unique self-adoint extension) consists of convolutions with probability measures.
Since $L$ is hypoelliptic, these measures are absolutely continuous with respect to the Haar measure and have densities in the Schwartz space \cite{FS}, which are strictly positive by the maximum principle \cite{Bony}.
Hence for all $t \in \RR^+$, there exists $p_t\in \cS(G)$ such that
\begin{equation*}
e^{-tL}f (x) =  f  * p_t (x) =  \int_G f (xy^{-1}) \fn p_t (y) \,dy
\end{equation*}
for all $x \in G$ and all $f \in L^2(G)$.
The function $p_t$ is called the heat kernel of $L$.
By the homogeneity of $L$ with respect to the dilations $\delta_r$,
\begin{equation*}
p_t (x) = t^{-{\sfrac Q2}} \fn P(\delta_{t^{-\sfrac12}} x)
\end{equation*}
for all $x \in G$, where $P=p_1$; this is a strictly positive function in $\cS(G)$ and
\begin{equation*}
\int_G P(x) \wrt x = 1 \ .
\end{equation*}
The associated heat equation has been extensively studied in the more general contexts of Lie groups (see, for example \cite{VSC}) and manifolds (see, for example \cite{Gilkey}), and much of what applies in the case of stratified Lie groups also applies there, but some aspects of our environment are special; for instance, Schwartz spaces and dilations are not defined in general.

We define the fractional integral kernel $F_\alpha$ on $G \setminus\{0\}$ when $\Re \alpha < Q$ by the formula
\begin{equation}\label{Ialpha}
F_{\alpha} (x)
= \int_0^{\infty} t^{\sfrac \alpha2} \, p_t (x)  \,\frac{dt}{t}
= \int_0^{\infty} t^{\sfrac{(\alpha - Q)} 2} \, P \lpar \delta_{t^{-1/2}} x\rpar  \,\frac{dt}{t} \ ;
\end{equation}
the integral converges absolutely and uniformly on compact subsets of $G\setminus\{0\}$ to a smooth function, homogeneous of degree $\alpha - Q$.
Moreover, $F_{\alpha}$ is positive when $\alpha$ is real, since $P$ is positive.
Hence the function $\absdot_{\alpha}$, given by
\begin{equation}\label{normal}
\labs  x \rabs_{\alpha} =
\begin{cases}
\lpar  F_{0} (x) \rpar ^{-\sfrac 1 {(Q-\alpha)}} &\text{ if  $x \in G \setminus\{0\} $}  \\
 0 &\text{ if $x=0$} \ ,
\end{cases}
\end{equation}
is  nonnegative and homogeneous of degree $1$, and vanishes only at the origin, so is a homogeneous norm in the sense of Folland and Stein \cite{FS}.
In general, these norms are all different.
Not all homogeneous norms on stratified groups are subadditive, so some authors use the term gauge, but for some cases where this holds, see \cite{Cygan}.

We shall use the norm $\absdot_{0}$.
Given any homogeneous norm $\absdot$ on $G$, there are positive constants  $A$ and $B$ such that, for all $x\in G$,
\begin{equation}\label{norm}
A \labs x\rabs_{0} \leq  \labs x\rabs \leq B \labs x\rabs_{0}
\end{equation}
for all $x\in G$.

Denote the half-plane $\{\alpha \in \CC : \Re \alpha < Q\}$ by $\cH_Q$.
As a function of $\alpha$, the integral $F_{\alpha} (x)$ in \eqref{Ialpha} is holomorphic in $\cH_Q$, and its derivatives $F_{\alpha}^{(k)} (x)$ are given by the absolutely convergent integrals
\begin{equation*}
F_{\alpha}^{(k)} (x) = \frac{1}{{2^k}} \int_0^{\infty} t^{\sfrac \alpha 2}  (\log t)^{k} \, p_t (x)  \,\frac{dt}{t}\ .
\end{equation*}

\begin{lemma}\label{stime puntuali per le derivate di I-alpha}
Let $\absdot$ be a homogeneous norm on $G$.
If $x \neq 0$ and $ \alpha \in \cH_Q$, then $F_{\alpha} (x)$ and its derivatives  $F_{\alpha}^{(k)} (x)$  in $\alpha$ satisfy the estimates
\begin{equation*}
\labs  F_{\alpha}^{(k)} (x) \rabs
\leq C_{\alpha,k} \labs x\rabs^{\Re \alpha -Q} \lpar 1 + \biglabs  \log \labs x \rabs \bigrabs^{k}
\rpar ,
\end{equation*}
where $k = 0, 1, \dots$, and the constants $C_{\alpha,k}$ are independent of $\Im\alpha$.
\end{lemma}

\begin{proof}
This follows immediately from the homogeneity of $F_{\alpha}^{(k)} (x)$ in $x$.
Indeed, differentiating the identity
\begin{equation*}
F_\alpha(\delta_rx)=r^{\alpha-Q}F_\alpha(x)
\end{equation*}
with respect to $\alpha$, we see that
\begin{equation*}
F_{\alpha}' ( \delta_{r} x )
= r^{\alpha - Q} F_{\alpha}' (x) + (r^{\alpha - Q} \log r) \, F_{\alpha} (x)\ ,
\end{equation*}
and, inductively,
\begin{equation*}
F_\alpha^{(k)}(\delta_rx)=r^{\alpha - Q}\sum_{j\le k}c_{kj}(\log r)^{k-j}F_\alpha^{(j)}(x)\ .
\end{equation*}
Write $x = \delta_{\labs x\rabs} x'$, where $\labs x' \rabs = 1$; then this identity yields
\begin{equation*}
F_\alpha^{(k)} ( x ) = \labs x\rabs^{\alpha - Q}\sum_{j\le k}c_{kj}(\log \labs x\rabs)^{k-j}F_\alpha^{(j)}(x')\ ,
\end{equation*}
which implies that
\begin{align*}
\labs  F_\alpha^{(k)} ( x ) \rabs
&\leq
\bigglpar \sup_{\labs x' \rabs = 1}\sum_{j\le k}\labs c_{kj} \rabs\labs  F_{\alpha}^{(j)} (x') \rabs\biggrpar \labs x\rabs^{\Re \alpha - Q} \lpar  1 + \biglabs  \log \labs x\rabs \bigrabs^{k}\rpar  \ ,
\end{align*}
as required.
\end{proof}

From this, we can estimate the remainder term in the Taylor series in $\alpha$ of~$F_\alpha(x)$.

\begin{proposition}
\label{sviluppo del primo ordine}
For all $\alpha\in \cH_Q$, there is a constant $C_\alpha$, independent of $\Im\alpha$, such that
\begin{equation*}
\begin{gathered}
\labs  F_{\alpha + \tau} (x)  -  F_{\alpha} (x) \rabs
\leq
 C_\alpha
{\labs \tau \rabs}
\,
 \labs x\rabs^{\Re \alpha - Q}
 \lpar  1 +
\labs  \log \labs x\rabs\rabs
 \rpar
  \lpar  1 + \labs x\rabs^{\Re \tau}
 \rpar
\\
\labs  F_{\alpha + \tau} (x)  -  F_{\alpha} (x) - {\tau}  F_{\alpha}' (x)  \rabs
\leq
C_\alpha
{\labs \tau \rabs^2}  \labs x\rabs^{\Re \alpha - Q}
 \lpar  1 +  \log^{2} \labs x\rabs \rpar
  \lpar  1 + \labs x\rabs^{\Re \tau}  \rpar
 \end{gathered}
 \end{equation*}
whenever $x \in G \setminus\{0\}$ and $\alpha + \tau \in \cH_Q$.
In particular, for $\alpha$ close to $0$,
\begin{equation}\label{F_02}
F_{\alpha} (x) =F_0(x)+\alpha F_0'(x)+O(\alpha^2) \ ,
\end{equation}
uniformly on compact subsets of $G\setminus\{0\}$.
\end{proposition}

\begin{proof}
To prove the second  inequality, we use the identity
\begin{align*}
\label{stima puntuale 1}
F_{\alpha + \tau} (x)  -  F_{\alpha} (x) - {\tau}  F_{\alpha}' (x) & =\tau^2\int_0^1(1-s) \, F_{\alpha+s\tau}''(x)\,ds\ ,
\end{align*}
and apply Lemma \ref{stime puntuali per le derivate di I-alpha}.
The first inequality may be proved similarly.
\end{proof}

Since the functions $F_{\alpha}$ are smooth and homogeneous of degree $\alpha - Q$, they are locally integrable on $G$ when $0 < \Re \alpha < Q$ and thus define distributions by integration:
\begin{equation}\label{distrib}
\begin{aligned}
\lip F_\alpha,\phi\rip
&=\int_G F_\alpha(x) \, \phi(x) \wrt x\\
&=\int_G\phi(x)\int_0^\infty t^{(\alpha-Q)/2} \, P(\delta_{t^{-1/2}}x)  \,\frac{dt}t \wrt x\\
&=2\int_0^\infty s^\alpha\int_G \phi(\delta_sx) \, P(x) \wrt x \,\frac{ds}s\
\end{aligned}
\end{equation}
for all $\phi \in C^\infty_c(G)$.
We write $I_\alpha$ for $\Gamma( \alpha/2)^{-1} F_\alpha$.
From \eqref{Ialpha}, we see that  $I_\alpha$ is the \emph{Riesz potential of order $\alpha$}, that is, the convolution kernel of $L^{-\sfrac\alpha 2}$, a power of the sub-Laplacian, and
\begin{equation}\label{Lpowers}
L^{-{\sfrac \alpha 2}}f  = f * I_\alpha\ .
\end{equation}
Note that
\begin{equation}\label{L^0}
{ I_\alpha} \bigrist{\alpha=0}=\delta_0\ .
\end{equation}

The family of Riesz potentials may be analytically continued as distributions to the half-plane $\cH_Q$ by the identity
\begin{equation*}
 I_\alpha  =  L \lpar  I_{\alpha+2} \rpar  \ .
\end{equation*}
The  analytic continuation to the strip $\{ \alpha \in \CC : -1<\Re\alpha<Q\}$ will be enough for our purposes.
An explicit expression is obtained by rewriting \eqref{distrib} in the form
\begin{equation}\label{I con s}
\begin{aligned}
 \lip I_{\alpha}, \phi \rip
&= \frac{2}{\Gamma(\frac{1}{2} \alpha)} \int_G  \lpar  \int_0^{1}  s^{\alpha} \lpar   \phi (\delta _{s} x) - \phi(0) \rpar    \,\frac{ds}{s} \rpar  P ( x) \wrt x
		+ \frac{2} {\alpha  \Gamma(\frac{1}{2} \alpha)} \, \phi (0)
\\
&\qquad +
 \frac{2}{{\Gamma(\frac{1}{2} \alpha) }}
  \int_G
\lpar
 \int_1^{\infty}  s^{\alpha} \,  \phi (\delta _{s} x)  \,\frac{ds}{s}
  \rpar
 P ( x) \wrt x \ ;
 \end{aligned}
 \end{equation}
the mean value theorem for stratified groups \cite{FS} ensures that the first integral converges.

Next, we find the Taylor expansion of $I_\alpha$ around $0$.

\begin{proposition}
Let  $\phi$ in $C_c^{\infty} (G)$.
For $\alpha$ near $0$,
\begin{align}\label{sviluppo di F_alpha}
 \lip I_{\alpha}, \phi \rip
&= \phi (0) +  \lip \Lambda, \phi \rip  {\alpha} + O(\alpha^2) \ ,
 \end{align}
 where ${\Lambda}$ is the distribution
 defined by
 \begin{equation}\label{Lambda}
 \begin{aligned}
\lip \Lambda, \phi  \rip
&=
\int_G
\lpar
 \int_0^{1}
\lpar   \phi (\delta _{s} x) - \phi(0) \rpar
    \,\frac{ds}{s}
  \rpar
 P ( x) \wrt x
 +\frac\gamma2 \,\phi(0) \\
&\qquad +
\int_G
\lpar
 \int_1^{\infty}
   \phi (\delta _{s} x)
    \,\frac{ds}{s}
  \rpar
 P ( x) \wrt x\ ,
\end{aligned}
\end{equation}
$\gamma$ being the Euler--Mascheroni constant.
\end{proposition}

\begin{proof}
Clearly, \eqref{Lambda} defines a distribution.
Since $\Gamma'(1)=-\gamma$ and hence
\begin{equation}\label{1/Gamma}
\frac1{\Gamma(t)}
=\frac t{\Gamma(t+1)}=t-\gamma t^2+O(t^3)\quad \text{as $t\to0$} \ ,
\end{equation}
and
\begin{equation*}
s^t
=1+\log s\int_0^t s^u\,du\ ,
\end{equation*}
\eqref{sviluppo di F_alpha} follows from \eqref{I con s}.
\end{proof}

Note that, in analogy with \eqref{Lpowers},
\begin{equation*}
f * \Lambda = -\log L^{\sfrac12} f \ .
\end{equation*}

The next proposition expresses the distribution $\Lambda$ in terms of the homogeneous norm $\absdot_0$, defined in \eqref{normal}.

\begin{proposition}
\label{lambda 2}
There is a positive number $a$ such that
\begin{equation}\label{lambda tilde}
\lip {\Lambda}, \phi \rip
=\frac 12
\int_{\{x \in G: \labs x\rabs_0 < a \}}
(\phi (x) - \phi(0) )
 \labs x\rabs_0^{-Q}
 \wrt x
 +
 \frac 12
  \int_{\{x \in G: \labs x\rabs_0 > a \}}
  \phi ( x) \labs x\rabs_0^{-Q}
 \wrt x\  ,
 \end{equation}
 for all  $\phi \in C^\infty_c$.
 \end{proposition}

\begin{proof}
We start from a different expression for the analytic continuation  \eqref{I con s} of the Riesz potentials.

Fix a positive number $a$, to be specified later.
If $\phi  \in C_c^{\infty} (G)$ and $0 < \Re \alpha < Q$, then
\begin{equation}\label{I con s*}
\begin{aligned}
 \lip I_{\alpha}, \phi \rip
 &=
\frac{1}{{\Gamma(\frac{1}{2} \alpha)}} \int_G \phi (x)  \,F_{\alpha} (x)
\, \wrt x
\\[.5ex]
&=   \frac{1}{{\Gamma(\frac{1}{2} \alpha)}}
\int_{\{x \in G: \labs x\rabs_0 < a \}}
 \lpar   \phi (x) - \phi(0) \rpar
F_{\alpha} (x)
 \wrt x
\\[.5ex]
&\qquad+  \frac{ \phi(0)} {\Gamma(\frac{1}{2} \alpha)}
\int_{\{x \in G: \labs x\rabs_0 < a \}}  F_{\alpha} (x) \wrt x
+ \frac{1}{{\Gamma(\frac{1}{2} \alpha) }}  \int_{\{x \in G: \labs x\rabs_0 > a \}}  \phi ( x) F_{\alpha} (x) \wrt x\\[.5ex]
&=T_1(\alpha)+T_2(\alpha)+T_3(\alpha)\ ,
 \end{aligned}
\end{equation}
say.
Clearly, $T_1$ is defined when $-1 < \Re \alpha < Q$, by the stratified mean value theorem, and $T_3$ is trivially defined for every $\alpha$. Moreover $T_1$ and $T_3$  are analytic.

To treat $T_2$, we use polar coordinates.
Write $x = \delta_r x'$, where $\labs x\rabs_0 = r$ and $x'$ is in $S$, the unit sphere for the norm $\absdot_0$,.
We denote  by $\sigma$ the unique measure  on $S$  (see \cite{FS})  such that
\begin{equation*}
\int_{\{x \in G: \labs x\rabs_0 < R \}}
\phi (x)
 \wrt x
 =
 \int_0^R
 \lpar
  \int_S
\phi ( \delta_r x')
 \,d\sigma  (x')
 \rpar
  r^{Q - 1} \,dr\ .
\end{equation*}
Since $F_{\alpha} (\delta_r x) = r^{\alpha - Q} F_{\alpha} ( x)$, it follows 
 that
\begin{equation*} 
\int_{\{x \in G: \labs x\rabs_0 < a \}}
F_{\alpha} (x)
 \wrt x
 =
 \int_0^a
 \lpar
  \int_{S}
 F_{\alpha} ( x')
 \,d\sigma  (x')
 \rpar
  r^{\alpha - 1} \,dr
  =
  c_{\alpha} \, \frac{a^{\alpha}}{\alpha} \ ,
\end{equation*}
where
\begin{equation*}
c_{\alpha}
=  \int_{S}
 F_{\alpha} (x')
 \,d\sigma (x') \ .
\end{equation*}
By \eqref{F_02},
\begin{equation*}
 F_{\alpha} (x') = 1+ \alpha F_0' (x')
 + O(\alpha^2) \ ,
\end{equation*}
for small $\alpha$, where $O(\alpha^2)$ is uniform in $x'$ in $S$.
This gives
\begin{equation*} 
c_{\alpha}
=  c_{0} +  c'_0 \alpha + O(\alpha^2) \ ,
\end{equation*}
where $c_{0} =  \sigma (S)$ and
\begin{equation*}
c'_0
=  \int_S
F_{0}' (x')
 \,d\sigma (x') \ .
\end{equation*}
Hence, using \eqref{1/Gamma},
\begin{equation*}
T_2(\alpha)
= \frac{c_{\alpha} a^{\alpha}}  {\alpha \Gamma(\frac{1}{2} \alpha)} \,\phi(0)
= \frac{\phi(0)}{ 2} \lpar  c_{0} + (c'_0 + c_{0} \log a - c_{0} \gamma) \alpha  + O(\alpha^2) \rpar  .
\end{equation*}
Since $c_0 \ne 0$, we may choose $a$ so that the linear term in the above formula vanishes.
Plugging this expansion  into \eqref{I con s*}, we deduce that
\begin{equation}\label{I con | |}
\begin{aligned}
 \lip I_{\alpha}, \phi \rip &=  \frac{1}{{\Gamma(\frac{1}{2} \alpha)}}
\int_{\{x \in G: \labs x\rabs_0 < a \}}
 \lpar   \phi (x) - \phi(0) \rpar
F_{\alpha} (x)
 \wrt x
\\
&\qquad+\lpar  \frac{c_0}{ 2} +O(\alpha^2)\rpar \phi (0)+
  \frac{1}{{\Gamma(\frac{1}{2} \alpha)}}
  \int_{\{x \in G: \labs x\rabs_0 > a \}}
  \phi ( x) \fn F_{\alpha} (x)
 \wrt x \ ,
 \end{aligned}
 \end{equation}
 which, compared with \eqref{L^0}, yields
\begin{equation}
\label{c-0}
c_0 = 2 \ .
\end{equation}
The linear term in $\alpha$ in \eqref{I con | |} gives $\lip\Lambda,\phi\rip$.
The result follows from \eqref{1/Gamma}.
\end{proof}

When $\Re \alpha>0$, $\Re \beta>0$ and $\Re (\alpha + \beta) < Q$, the convolution of the Riesz potentials of orders $\alpha$ and $\beta$ is defined pointwise by absolutely convergent integrals, and the identity
\begin{equation}\label{algebra}
I_{\alpha} * I_{\beta}   = I_{\alpha + \beta}
\end{equation}
is satisfied pointwise and in the sense of distributions.
This is consistent with the functional identity $L^{-\sfrac\alpha2}L^{-\sfrac\beta2}=L^{-\sfrac{(\alpha+\beta)}2}$
and can be derived from \eqref{Ialpha} and the properties of the heat kernel.

We extend this result using analytic continuation.

\begin{proposition}\label{convI}
Let $\phi\in C^\infty_c(G)$.
The identity
\begin{equation*}
\lpar \phi* I_{\alpha} \rpar  * I_{\beta}  = \phi* I_{\alpha + \beta}
 \end{equation*}
 holds if $\Re \alpha >-1$, $\Re \beta>-1$ and $\Re (\alpha + \beta) < Q$.
 \end{proposition}

 \begin{proof}
 We must show that the left-hand side is well-defined and analytic for $\alpha$ and $\beta$ in the given range.

Assume that $\beta\ne0$.
The inner convolution is well-defined and produces a smooth function $u_\alpha$, satisfying $u_\alpha(x) = O( \labs x\rabs_0^{-Q+\Re\alpha})$ at infinity.
By \eqref{I con s*},
\begin{equation*}
 \begin{aligned}
 \lpar \phi* I_{\alpha} \rpar  * I_{\beta}(x)
 &=
 u_\alpha* I_{\beta} (x) \\
 &=
\int_{\{y \in G: \labs y \rabs_0 < a \}}
 \lpar   u_\alpha(xy^{-1}) - u_\alpha(x) \rpar
I_{\beta} (y)
 \,dy
\\
&\qquad+  u_\alpha(x)
\int_{\{y \in G: \labs y \rabs_0 < a \}}
I_{\beta} (y)
 \,dy
+
  \int_{\{y \in G: \labs y \rabs_0 > a \}}
  u_\alpha( xy^{-1})
\fn I_{\beta} (y)
 \,dy
 \end{aligned}
\end{equation*}
for all $x \in G$; the  integrals are absolutely convergent and each term depends analytically on $\alpha$ and $\beta$. The conclusion follows from \eqref{algebra} by analytic continuation.
 \end{proof}

\section{Hardy inequalities}

In this section, we use the distributions discussed above to obtain a family of Hardy-type inequalities.

Take a homogeneous norm $\absdot$ on $G$.
For $-1 < \Re\alpha < Q$, we define the operator $T_{\alpha}$ on $C^{\infty}_c (G)$ by
\begin{equation*}
T_{\alpha} f
=  \absdot^{- \alpha} L^{- \sfrac{\alpha}{2}} f = \absdot^{- \alpha}(f*I_\alpha) \ .
\end{equation*}
The operator $T^{*}_{\alpha}$, given by
\begin{equation*}
T^{*}_{\alpha} g
=  (\absdot^{- \bar\alpha} g)  * I_{\bar\alpha} \ ,
\end{equation*}
satisfies the identity
\begin{equation}\label{Tstar}
\lip f,T^{*}_{\alpha} g\rip=\lip T_{\alpha} f,g\rip
\end{equation}
for all $f,g\in C^\infty_c(G)$.
The right-hand side of \eqref{Tstar} is holomorphic in $\alpha$ and $T_0$ is the identity.

In this section we prove Theorem A, of which  we repeat the statement for the reader's convenience.
\begin{theorem*}
Suppose that $\absdot$ is a homogeneous norm on $G$, that $1 < p < \infty$ and that $0 < \alpha < \sfrac{Q}{p}$.
Then  the operator $T_{\alpha}$ extends uniquely to a bounded operator  on $L^p(G)$.
For the particular homogeneous norm $\absdot_0$, the operator norm $\lopnorm T_\alpha \ropnorm_{p,p}$ satisfies
\[
\lopnorm T_\alpha \ropnorm_{p,p} \leq 1 + C \alpha + O(\alpha^2)\ .
\]
\end{theorem*}

We first prove that $T_\alpha$ is bounded on $L^p(G)$ when $1<p<\infty$ and $0<\alpha<\sfrac Qp$.

\begin{theorem}\label{Talpha}
Let $1 < p < \infty$.
If $0 < \alpha < \sfrac{Q}{p}$, then  the operator $T_{\alpha}$ extends uniquely to a bounded operator  on $L^p(G)$.
\end{theorem}

\begin{proof}
Fix $p$ and $\alpha$ as enunciated, and let $p'$ denote the conjugate index to $p$.
Since the integral kernel of $T_\alpha$ is positive, we can apply  Schur's test \cite{FR}.
It suffices to exhibit a positive function $u$  and constants $A_{\alpha, p}$ and $B_{\alpha, p}$ such that
\begin{equation}\label{schur}
T_{\alpha} (u^{p'})(x)  \leq A_{\alpha, p} u^{p'} (x)
\quad\text{and}\quad
T^{*}_{\alpha}  (u^{p})(x)  \leq B_{\alpha, p} u^{p}(x)
\end{equation}
for almost all $x \in G$ .
It then follows that, for all $f\in L^p(G)$,
\begin{equation}\label{schur-constant}
\lnorm T_\alpha f \rnorm_p\le {A_{\alpha, p}^{1/p'}  B_{\alpha, p}^{1 / p}}\lnorm f \rnorm_p\ .
\end{equation}

For $\gamma>0$, take $u_{\gamma} = \absdot^{\gamma-Q}$ and consider the convolution integrals
\begin{equation*}
u_\gamma^{p'}*I_{\alpha}
\quad\text{and}\quad
 (\absdot^{-\alpha} u_\gamma^{p}) * I_{\alpha}
\end{equation*}
involved in the computation of $T_\alpha(u_\gamma^{p'})$ and $T_\alpha^*(u_\gamma^p)$.
We set
\begin{equation}\label{beta-beta'}
\beta'=Q+(\gamma-Q)p'
\quad\text{and}\quad
 \beta=Q-\alpha+(\gamma-Q)p\ ,
\end{equation}
so $u_\gamma^{p'}$ and $\absdot^{-\alpha}u_\gamma^{p}$ have the same homogeneity as $I_{\beta'}$  and $I_\beta$.
As in the proof of \eqref{algebra}, the convolution integrals converge absolutely in $G \setminus\{0\}$ if and only if $0<\beta<Q-\alpha$ and $0<\beta'<Q-\alpha$, that is,
for $\gamma$ such that
\begin{equation}\label{gamma}
\max \lpar \frac{Q}{p}, \frac{\alpha}{p} + \frac{Q}{p'} \rpar  < \gamma <  Q - \frac{\alpha}{p'} \ .
\end{equation}
When this condition is satisfied, $T_\alpha(u_\gamma^{p'})$ and $T_\alpha^*(u_\gamma^p)$ are positive  functions, continuous away from 0, and with the same homogeneity as $I_{\beta'}$ and $I_\beta$.
The pointwise estimates \eqref{schur} follow trivially by homogeneity.

The range \eqref{gamma} for $\gamma$ is nontrivial if and only if  $0 < \alpha < \sfrac{Q}{p}$, and the proof is complete.
 \end{proof}

 Note that if $\alpha \geq \sfrac{Q}{p}$, then $T_\alpha f$ can be infinite everywhere; however, if $\alpha$ is complex and $0 \leq \Re \alpha < \sfrac{Q}{p}$, then $T_\alpha$ is bounded on $L^p(G)$, but it is harder to control the norm.

 \begin{corollary}\label{hardy}
 Let $1 < p < \infty$ and $0 < \alpha < \sfrac{Q}{p}$.
For all $f\in C_c^\infty(G)$,
\begin{equation*}
\lnorm  \absdot^{-\alpha}f\rnorm_p\le \lopnorm T_\alpha \ropnorm_{p,p} \lnorm  L^\sfrac\alpha2 f\rnorm_p\ .
\end{equation*}
\end{corollary}

\begin{proof}
We apply Theorem \ref{Talpha} to $L^{\sfrac\alpha2}f$ rather than $f$.
\end{proof}

For the rest of this section, we assume that $\absdot = \absdot_{0}$ and $T_\alpha$ is defined using $\absdot_{0}$.
The proof of Theorem \ref{Talpha}, and especially \eqref{schur-constant}, shows that, if
\begin{equation}\label{gamma-bounds}
 0<\alpha< \frac{Q}{ p}
 \quad\text{and}\quad \frac{\alpha}{p} + \frac{Q } {p'} < \gamma <  Q -  \frac{\alpha} {p'}\ ,
\end{equation}
then
\begin{equation}\label{gamma->T}
 \lopnorm T_\alpha \ropnorm_{p,p}\le {A_{\alpha,\gamma, p}^{1/p'}  B_{\alpha,\gamma, p}^{1 / p}}\ ,
 \end{equation}
where 
\begin{equation}\label{schur1}
\begin{aligned}
A_{\alpha, \gamma, p}
&=  \sup_{ x' \in S} \lpar \absdot_{0}^{(\gamma-Q) \, p'} * I_{\alpha} \rpar  (x')
=   \sup_{ x' \in S} \lpar \absdot_{0}^{\beta'-Q} * I_{\alpha} \rpar  (x')\ ,
\\ 
B_{\alpha, \gamma, p}
&=  \sup_{x' \in S} \lpar  \absdot_{0}^{(\gamma-Q) \, p - \alpha} * I_{\alpha} \rpar  (x')
= \sup_{x' \in S} \lpar  \absdot_{0}^{\beta-Q} * I_{\alpha} \rpar  (x')\ .
\end{aligned}
\end{equation}
Estimating these quantities is the focus of our next result.
Recall the distribution $\Lambda$ introduced in \eqref{Lambda}.
We also define the operator $T'_0$ by
\begin{equation}\label{T'_0}
T'_0f =\Biglpar \frac d{d\alpha}T_{\alpha} f \Bigrpar \Bigrist{\alpha = 0}=-\log L^{\sfrac12} f- (\log\absdot_0) f =  - f * \Lambda -(\log\absdot_0) f\ .
\end{equation}
We also look for bounds on the operator norm $\lopnorm T_\alpha \ropnorm_{p,p}$, of the form
\begin{equation*}
\lopnorm T_\alpha \ropnorm_{p,p}\le 1+C_p\alpha\ ,
\end{equation*}
for $\alpha$ close to 0.
These bounds will lead to our logarithmic uncertainty inequality for $T'_0$.

\begin{lemma}\label{calC}
Let $0 < \beta < Q $.
If $\alpha$ is small and $0<\alpha<Q-\beta$, then
\begin{equation*}
\sup_{x'\in S}\lpar \absdot_0^{\beta - Q} *I_{\alpha}\rpar   (x')
\le 1+L_\beta\alpha + O(\alpha^2)\ ,
\end{equation*}
where
\begin{equation}\label{L_beta}
L_\beta=\sup_{x'\in S} \lpar  \absdot_0^{\beta - Q}*\Lambda(x')\rpar \ .
\end{equation}
\end{lemma}

\begin{proof}
The integrals in \eqref{lambda tilde} converge absolutely when $\phi$ is replaced by $\labs x'{\cdot} \rabs_0^{\beta-Q}$ for any $x'\in S$, and so \eqref{sviluppo di F_alpha} also extends, that is,
\begin{equation*}
\absdot_0^{\beta - Q} *I_{\alpha} (x')
= 1+\alpha \absdot_0^{\beta - Q}*\Lambda(x') + O(\alpha^2)\ ;
\end{equation*}
the $O(\alpha^2)$ term is uniform in $x' \in S$.
Taking the supremum over $x'$ yields the result.
\end{proof}

We repeat that we are now specializing to the particular homogeneous norm $\absdot_{0}$.

\begin{corollary}\label{finalmente}
Let $1 < p <\infty$.
There is a constant $C_p$, such that, for all sufficiently small positive $\alpha$,
\begin{align}\label{finale}
\lopnorm  T_{\alpha} f  \ropnorm_{p,p}
\leq  \lpar 1 + {C_{p}} \,{\alpha}+O(\alpha^2)\rpar \ .
\end{align}
\end{corollary}

\begin{proof}
Fix $\gamma$ satisfying $\sfrac{Q} {p'} < \gamma <  Q$.
Then \eqref{gamma-bounds} holds for $\alpha$ in a right neighborhood of 0.
With $\beta$ and $\beta'$ as in \eqref{beta-beta'}, the constants $A_{\alpha, \gamma, p}$ and $B_{\alpha, \gamma, p}$ in \eqref{schur1} are bounded by   $1+L_{\beta'}\alpha+O(\alpha^2)$ and $1+L_\beta\alpha+O(\alpha^2)$.
By \eqref{gamma->T},
\begin{equation*}
\lopnorm T_\alpha \ropnorm_{p,p}\le 1+\lpar \frac{L_{\beta'}}{p'}+\frac{L_\beta}p\rpar \alpha+O(\alpha^2)\ ,
\end{equation*}
as required.
\end{proof}

Note that $\Lambda$ is not a positive distribution, hence the constants $L_\beta$ in \eqref{L_beta} and $C_p$ in \eqref{finale} need not be positive.

\section{A logarithmic uncertainty inequality}\label{sec:log}

In this section, we find logarithmic versions of the Heisenberg uncertainty principle.
We were inspired by the work of W.~Beckner \cite{Beck1, Beck2}, but unlike Beckner, we do not look for sharp constants.
The following statement is Theorem B of the Introduction.

\begin{theorem*}\label{logarithmic}
Let $1 < p < \infty$. Then, for all $f\in C^{\infty}_c(G)$,
\begin{equation}\label{log2}
\int_{G}  (\log \labs x\rabs) \labs f (x) \rabs^{p} \wrt x + \int_{G} \Re \biglpar (\log L^\sfrac12 f)(x) \,\overline{f (x) } \bigrpar  \labs f (x) \rabs^{p - 2} \wrt x
 \ge C\lnorm f \rnorm_p^p\ .
\end{equation}
\end{theorem*}

\begin{proof}
First, we consider the particular homogeneous norm $\absdot_{0}$.
Taking $f\in C^{\infty}_c(G)$ with $\lnorm f \rnorm_p=1$ and restricting ourselves  to positive values of $\alpha$ for which \eqref{finale} holds and $1+C_p\alpha>0$, we consider the function
\begin{equation*}
\Phi_\epsilon(\alpha)=(1+(C_p+\epsilon)\alpha)^p-\lnorm T_\alpha f \rnorm_p^p\ ,
\end{equation*}
where $\epsilon\geq 0$.
The expression $\lnorm T_\alpha f \rnorm_p^p$ can be differentiated in $\alpha$ at $0$, and
\begin{equation*}
\frac d{d\alpha}\lpar \lnorm T_\alpha f \rnorm_p^p\rpar \rist{{\alpha=0}}
= p \int_G\Re\lpar T'_0 f(x) \,\overline{f (x)} \rpar  \labs f (x) \rabs^{p - 2} \wrt x\ .
\end{equation*}
Hence $\Phi_\epsilon$ is differentiable at $0$.
Now $\Phi_\epsilon(0)=0$ and $\Phi_\epsilon(\alpha)\ge0$ for all sufficiently small $\alpha$, by \eqref{finale}, and so $\Phi_\epsilon'(0)\ge 0$.
Now we let $\epsilon$ tend to $0$, and deduce that $\Phi_0'(0)\ge 0$.
It follows from \eqref{T'_0} that
\[
\int_{G}  (\log \labs x\rabs) \labs f (x) \rabs^{p} \wrt x + \int_{G} \Re \biglpar (\log L^\sfrac12 f)(x) \,\overline{f (x) } \bigrpar  \labs f (x) \rabs^{p - 2} \wrt x
 \ge -C_{p}\lnorm f \rnorm_p^p\ .
\]

For general homogeneous norms, the results follows using the equivalence \eqref{norm}.
\end{proof}

\begin{remark} When $\lnorm f \rnorm_p=1$, Jensen's inequality coupled with inequality \eqref{log2} yield
\begin{align*}
\frac{1}{r} \log \int_{G}  \labs x\rabs^r \labs f (x) \rabs^{p} \wrt x
 +\int_{G}  \Re \lpar  (\log L^\sfrac12 f)(x) \,\overline{f (x) }\rpar  \labs f (x) \rabs^{p - 2} \wrt x
\ge C\ ,
\end{align*}
for every $r>0$.
If we could find a similar way to move the logarithm outside the second integral, then we would be able to derive a version of Corollary \ref{finalmente} for an arbitrary homogeneous norm.
It is not clear to us whether this is possible.
\end{remark}

\section{Uncertainty inequalities}\label{sec:HPW_C}

In \cite{CRS}, a general Heisenberg--Pauli--Weyl uncertainty inequality was proved, in a broad setting that includes the case of a homogenous sub-Laplacian $L$ on a stratified group $G$.
It follows that
\begin{equation*}
\lnorm f \rnorm_2
\le C_{\beta,\gamma}\biglnorm  \absdot^\beta f\bigrnorm_2^\sfrac\gamma{(\beta+\gamma)}\biglnorm  L^\sfrac\gamma2 f\bigrnorm_2^\sfrac\beta{(\beta+\gamma)}
\end{equation*}
for all positive $\beta$ and $\gamma$ and all $f \in C^\infty_c(G)$.

In this section, we extend this inequality to more general $L^p$-norms.
To do this, we recall a weighted version of H\"older's inequality, then prove a corollary.
Next, we extend the Landau--Kolmogorov inequality, and finally we prove another corollary.
These results actually hold in more general contexts than stratified groups.

\begin{lemma}\label{lem:weights}
Suppose that $\absdot$ is a homogeneous norm on $G$, that $0 \leq \theta \leq 1$ and $\alpha \geq 0$.
If $1 \leq p, q, r \leq \infty$, and
\[
\frac{1}{p} = \frac{\theta}{q} + \frac{1-\theta}{r} \ ,
\]
then
\[
\lnorm \absdot^{\alpha} f  \rnorm_p
\leq \biglnorm  \absdot^{\sfrac\alpha\theta} f \bigrnorm_q^\theta \, \biglnorm f \bigrnorm_r^{1-\theta}
\]
for all $f \in C^\infty_c(G)$.
\end{lemma}

\begin{proof}
Write
\begin{equation*}
\lpar \absdot^\alpha \labs f  \rabs \rpar^p=  \lpar \absdot^{\alpha}  \labs f  \rabs^{\theta} \rpar^p  \lpar \labs f  \rabs^{1-\theta} \rpar^p\ ,
\end{equation*}
then apply H\"older's inequality with index $s$ and take $p$th roots to get
\begin{equation*}
\lpar \int_G \lpar \labs x\rabs^\alpha \labs f(x) \rabs \rpar^p \wrt x\rpar ^{\sfrac1p}
\leq  \lpar \int_G  \lpar \labs x\rabs^{\alpha}  \labs f(x) \rabs^{\theta} \rpar^{ps} \wrt x\rpar^{\sfrac1{ps}}
		\lpar \int_G  \lpar \labs f(x) \rabs^{1-\theta}  \rpar^{ps'} \wrt x\rpar^{\sfrac1{ps'}}  \ .
\end{equation*}
If $s$ is chosen so that $q = \theta p s$, then $r = (1-\theta) ps'$, and we are done.
\end{proof}

Of course, it is routine to extend this result to more general measurable functions.

\begin{corollary}\label{HPW}
Suppose that $\beta>0$, $\gamma>0$, $p>1$, $q\geq 1$,  $r > 1$, and that
\begin{equation*}
\gamma <\frac Qr
\quad\text{and}\quad
\frac{\beta+\gamma}p = \frac\gamma q+\frac\beta r \ .
\end{equation*}
Then
\begin{equation}\label{HPWeq}
\lnorm f \rnorm_p
\le  C \biglnorm \absdot^\beta f\bigrnorm_q^\sfrac\gamma{(\beta+\gamma)} \lnorm  L^\sfrac\gamma2 f\rnorm_r^\sfrac\beta{(\beta+\gamma)}
\end{equation}
for all $f \in C^\infty_c(G)$.
\end{corollary}

\begin{proof}
Use H\"older's inequality, as in the proof of Lemma \ref{lem:weights}, Corollary \ref{hardy} with the estimate \eqref{finale}:
\[
\begin{aligned}
\lnorm f \rnorm_p
&\le  \biglnorm \absdot^\beta f\bigrnorm_q^\sfrac\gamma{(\beta+\gamma)}\lnorm  \absdot^{-\gamma}  f\rnorm_r^\sfrac\beta{(\beta+\gamma)} \\
&\le C    \biglnorm \absdot^\beta f\bigrnorm_q^\sfrac\gamma{(\beta+\gamma)}\lnorm  L^\sfrac\gamma2 f\rnorm_r^\sfrac\beta{(\beta+\gamma)}\ ,
\end{aligned}
\]
as required.
\end{proof}

If we use the homogeneous norm $\absdot_{0}$, then we can show that the constant $C$ in \eqref{HPW} may be taken to be $1+C_r \min\{\beta,\gamma\} + O\lpar(\min\{\beta,\gamma\}^2 \rpar $, with $C_r$ as in \eqref{finale}.
Since we can control the constants, we can differentiate the inequality of Corollary \ref{HPW} to obtain a logarithmic version.
At the cost of losing control of the constants, and hence the possibility of differentiating,  we can treat higher powers of the sub-Laplacian
 and obtain Theorem~D of the Introduction.

\begin{theorem*}\label{KL}
Suppose that $L$ is a sub-Laplacian on $G$, that $0 \leq \theta \leq 1$, and that $\alpha \geq 0$.
If $1 < p, q, r < \infty$, and
\[
\frac{1}{p} = \frac{\theta}{q} + \frac{1-\theta}{r} \ ,
\]
then
\[
\lnorm L^{\sfrac{\alpha}{2}} f  \rnorm_p \leq C \lnorm  L^{\sfrac{\alpha}{2\theta}} f \rnorm_q^\theta \lnorm f \rnorm_r^{1-\theta}
\]
for all $f \in C^\infty_c(G)$.
\end{theorem*}

\begin{proof}
We use the arguments of complex interpolation (see \cite{BL, Stein-interp}).
To do this, we need to know that the operators $L^{iy}$, where $y \in \RR$, are bounded on the Lebesgue spaces $L^s(G)$ when $1 < s < \infty$, and that $\lopnorm L^{iy/2} \ropnorm_{s,s} \leq C(s) \fn \phi(y)$, where $\phi(y) = e^{ \gamma\labs y\rabs }$.
For stratified groups, this follows from the Mihlin--H\"ormander multiplier theorem (see, for instance, \cite{Christ, MM}). In more general cases, such estimates follow from versions of this theorem for semigroups, see \cite{Carb-Drag, Cow, Stein-LPS}.

We write $S$ for the strip
\[
S = \Bigl\{ z \in \CC : 0 \leq \Re z \leq \frac{\alpha}{\theta} \Bigr\} \ .
\]

Take an arbitrary compactly supported simple function $g \in L^{p'}(G)$, and for $z$ in $S$, define $g_z:G \to \CC$ by
\[
g_z (x) = \lnorm g\rnorm_{p'} ^{az + b} g(x) \labs g(x) \rabs ^{cz +d}
\]
for all $x \in G$, where
\[
a = \frac{\theta p'}{\alpha}\lpar \frac{1}{q} - \frac{1}{p} \rpar, \quad b = -\frac{p'}{r'}, \quad c = \frac{\theta p'}{\alpha} \lpar \frac{1}{p} - \frac{1}{q} \rpar \quad\text{and}\quad d = \frac{p'}{r'} - 1 \ .
\]

When $\Re z = \eta$ and
\[
\frac{1}{s} = \frac{1}{r} - \frac{\theta \eta}{\alpha} \lpar \frac{1}{p} - \frac{1}{q} \rpar  ,
\]
it follows readily that
\[
\begin{aligned}
\lnorm g_z \rnorm_{s'}^{s'}
&=  \int_G \labs g_z(x) \rabs^{s'}  \wrt x
= \lnorm g\rnorm_{p'} ^{{s'}(a\eta + b) + p' }
= 1 \ .
\end{aligned}
\]

Now we fix $f \in C^\infty_c(G)$ and consider the analytic function $h : S \to \CC$, defined by
\[
h(z) = e^{z^2} \int_G L^\sfrac{z}{2} f (x) \, g_z(x) \, \wrt x \ .
\]

By our assumptions on $f$, for each $z \in S$, the function $L^\sfrac{z}{2}f$ on $G$ is smooth, while $g_z$ is a simple function with compact support, and so $h(z)$ is defined; moreover, if $z = \eta + iy$, then
\[
\begin{aligned}
\labs h(z) \rabs
&\leq e^{\eta^2 - y^2} \lnorm L^\sfrac{z}{2} f \rnorm_{s}  \lnorm g_z \rnorm_{s'}
= e^{\eta^2 - y^2}\lnorm L^\sfrac{iy}{2} L^{\sfrac{\eta}{2}} f \rnorm_{s}
\leq C e^{- y^2}\phi(y) \lnorm L^{\sfrac{\eta}{2}} f \rnorm_{s} \ ,
\end{aligned}
\]
whence $\labs h(\eta+iy) \rabs \leq C \lnorm L^{\sfrac{\eta}{2}} f \rnorm_{s}$.
Further, if $\Re z = 0$, then $\labs h(z) \rabs \leq C \lnorm f \rnorm_r$, while if $\Re z =  \sfrac{\alpha}{\theta}$, then $\labs h(z) \rabs \leq C \lnorm L^{\sfrac{\alpha}{2\theta}} f \rnorm_q$.
By the Phragm\'en--Lindel\"of 
theorem, it follows that
\[
\begin{aligned}
\labs h(\alpha) \rabs
&\leq C \lnorm L^{\sfrac{\alpha}{2\theta}} f \rnorm_q ^{\theta} \lnorm  f \rnorm_r ^{1-\theta} \ .
\end{aligned}
\]

Since
\[
h(\alpha) = \int_G L^\sfrac{\alpha}{2} f (x) \, g_\alpha(x) \, \wrt x \ ,
\]
and $g_\alpha$ is an arbitrary simple function on $G$ with compact support and $L^{p'}(G)$-norm equal to $1$, the desired estimate follows.
\end{proof}

It is more complicated to extend this inequality to more general functions, and we leave this to the interested reader.
See \cite{BL, SW} for more on the necessary technology.
Now we apply these interpolation results to extend our uncertainty inequalities to a wider range of Lebesgue indices and powers of the homogenous norm and sub-Laplacian, to obtain the following statement, which is Theorem C of the Introduction.

\begin{theorem*}\label{HPW2}
Suppose that $\beta>0$, $\delta>0$, $p>1$, $s\geq 1$,  $r > 1$, and that
\begin{equation*}
\frac{\beta+\delta}p = \frac\delta s+\frac\beta r \ .
\end{equation*}
Then
\begin{equation*}
\lnorm f \rnorm_p
\le C  \biglnorm \absdot_0^\beta f\bigrnorm_s^\sfrac\delta{(\beta+\delta)}\lnorm  L^\sfrac\delta2 f\rnorm_r^\sfrac\beta{(\beta+\delta)}
\end{equation*}
for all $f \in C^\infty_c(G)$.
\end{theorem*}

\begin{proof}
If $\delta < \sfrac Qr$, there is nothing to do.
Otherwise, we begin with the estimate from Corollary \ref{HPW}, and apply our version of the Landau--Kolmogorov inequality (Theorem D).
Take $\theta$ between $0$ and $\sfrac Q{(r\delta)})$, and let $\gamma = \theta\delta $.
Corollary \ref{HPW} shows that
\begin{equation*}
\lnorm f \rnorm_p
\le C  \biglnorm \absdot_0^\beta f\bigrnorm_q^\sfrac\gamma{(\beta+\gamma)}\lnorm  L^\sfrac\gamma2 f\rnorm_r^\sfrac\beta{(\beta+\gamma)}
\end{equation*}
for all $f \in C^\infty_c(G)$.
By Theorem \ref{KL}, with $r$, $s$ and $p$ in place of $p$, $q$, and $r$,
\[
\lnorm L^{\sfrac\gamma 2} f \rnorm_r \leq C \lnorm L^{\sfrac\delta 2}  f \rnorm_s^\theta \lnorm f \rnorm_p ^{1-\theta} \ ,
\]
and so
\begin{equation*}
\lnorm f \rnorm_p
\le C  \biglnorm \absdot_0^\beta f\bigrnorm_q^\sfrac\gamma{(\beta+\gamma)} \lnorm L^{\sfrac\delta 2}  f \rnorm_s^{\sfrac{\theta\beta}{(\beta+\gamma)}} \lnorm f \rnorm_p^{\sfrac{((1-\theta)\beta}{(\beta+\gamma)}} \ ,
\end{equation*}
and reorganizing this expression, we obtain the desired estimate.
\end{proof}

\section{An alternative approach to uncertainty inequalities}\label{sec:HPW_R}

We can use a real interpolation argument to give a simpler proof of an uncertainty inequality involving different Lebesgue indices; however the range of powers of the homogenous norm and of the sub-Laplacian that appear is limited.  Of course, we can always extend these using the arguments used in the previous section.

\begin{theorem}
Suppose that $1 < p,q, r< \infty$, that $0 < \alpha < Q/q'$, and that $0< \beta < Q/r$, and take $s$ and $t$ such that
\[
\frac{1}{s} =  \frac{1}{q} + \frac{ \alpha}{Q}
\quad\text{and}\quad
\frac{1}{t}  =  \frac{1}{r} - \frac{ \beta}{Q}  \ .
\]

Suppose also that either
\[
s \neq t
\quad\text{and}\quad
\frac{1}{p} = \frac{\theta}{s} + \frac{1-\theta}{t}  \ ,
\]
or
\[
s = t
\quad\text{and}\quad
\theta \leq \theta_0 = \frac{\beta}{\alpha + \beta}  \ .
\]

If $f \in C^\infty_c(G)$, then
\[
\lnorm f \rnorm_p \leq C \lnorm \absdot^\alpha f \rnorm_q^\theta  \lnorm L^{\beta/2} f \rnorm_r^{1-\theta} \ .
\]
\end{theorem}

\begin{proof}
We use Lorentz spaces and real interpolation \cite{BL}.
Recall that the nondecreasing rearrangement $f^*: \RR^+ \to [0, \infty)$ of a function $f$ on $G$ satisfies
\[
\labs \{ t \in \RR^+ :  f^*(t)  > \lambda \} \rabs = \labs \{ x \in G : \labs f(x) \rabs > \lambda \} \rabs
\]
for all $\lambda \in \RR^+$; the measures are Lebesgue measure on $\RR^+$ and the Haar measure on $G$.
The Lorentz space $L^{p,q}(G)$ is the set of all functions $f$ on $G$ such that $\lnorm f\rnorm_{p,q}$ is finite, where
\[
\lnorm f\rnorm_{p,q} = \lpar \int_0^\infty \labs t^{1/p} f^*(t) \rabs^q   \,\frac{dt}{t} \rpar ^{1/q},
\]
with the obvious modification if $q = \infty$.
The space $L^p(G)$ coincides with $L^{p,p}(G)$.
In general, $\lnorm\cdot\rnorm_{p,q}$ is a quasi-norm, not a norm.
Mostly there is an equivalent norm, though in some cases, $L^{p,q}(G)$ is not normable.

First, we observe that $ f = \absdot^{-\alpha} \absdot^\alpha f$, and study the operator $M$ of pointwise multiplication  by $\absdot^{-\alpha}$.
This function lies in $L^{ Q/\alpha ,\infty}(G)$, and so $M$ takes $L^{(Q/\alpha)', \infty}(G)$ to $L^{1,\infty}(G)$ and $L^\infty(G)$ to  $L^{ Q/\alpha ,\infty}(G)$.
By the real interpolation theorem, if $(Q/\alpha)' < q < \infty$, then $M$ also maps $L^{q,q}(G)$ to $L^{s,q}(G)$,  and
\[
\lnorm \absdot^{-\alpha} g \rnorm_{s,q} \leq C_{\alpha,q} \lnorm g \rnorm_{q,q} \ .
\]

It follows that 
\[
\lnorm f \rnorm_{s,q} \leq C_{\alpha,q} \lnorm \absdot^{\alpha} f \rnorm_{q} \ .
\]

Second, we observe that $f = L^{-\beta/2} L^{\beta/2} f$, and consider the convolution operator $L^{-\beta/2}$.
The kernel of this operator, $I_\beta$, lies in $L^{ Q/(Q-\beta), \infty}(G)$, and so the operator $L^{-\beta/2}$ takes $L^{1}(G)$ to  $L^{ Q/(Q-\beta), \infty}(G)$ and $L^{(Q/(Q-\beta))', 1}(G)$ to $L^{\infty}(G)$.
By the real interpolation theorem, if $1 < r < (Q/(Q-\beta))'$, then $L^{-\beta/2}$ also maps $L^{r,r}(G)$ to $L^{t,r}(G)$, and
\[
\lnorm  L^{-\beta/2} g \rnorm_{t,r} \leq C_{\beta,r} \lnorm g \rnorm_{r,r} \ .
\]

Returning to our problem, it follows that 
\[
\lnorm f \rnorm_{t,r} \leq C_{\beta,r} \lnorm L^{\beta/2} f \rnorm_{r} \ .
\]

To conclude, there are two cases to consider.
If $s \neq t$, then $p$ lies between these indices, and by hypothesis, $1/p = \theta/s + (1-\theta)/t$.
By the real interpolation theorem,
\[
\lnorm f \rnorm_p
\leq C \lnorm f \rnorm_{s,q}^\theta \lnorm f \rnorm_{t,r}^{1-\theta}
\leq C \lnorm \absdot^{\alpha} f \rnorm_{q} ^\theta \lnorm L^{\beta/2} f \rnorm_{r}^{1-\theta} \ .
\]

If $s = t$, then $p$ lies between the indices $q$ and $r$, and $p \geq u$, where $1/u = \theta/q + (1-\theta)/r$.
By complex interpolation between the $L^{p, q}$ and $L^{p, r}$ quasinorms,
\[
\lnorm f \rnorm_p
\leq \lnorm f \rnorm_{p,u}
\leq C \lnorm f \rnorm_{p,q}^\theta \lnorm f \rnorm_{p,r}^{1-\theta}
\leq C \lnorm \absdot^{\alpha} f \rnorm_{q} ^\theta \lnorm L^{\beta/2} f \rnorm_{r}^{1-\theta} \ ,
\]
and we are done.
\end{proof}

\begin{remark}
We have focussed on \emph{a priori} inequalities here.
The reader who wishes to prove more general inequalities will be able to do this by using the techniques of \cite{CowPrice}, coupled with the generalization to the stratified group setting in \cite{Ricci} of the Landau--Pollak--Slepian theorem (see \cite{LanPol, PolSle, Slep}) used in \cite{CowPrice}.
\end{remark}

\end{document}